\documentclass[a4paper]{amsart}
\usepackage{amsfonts,amsmath,amssymb,amsthm,cmtiup} 
\usepackage{mathrsfs}
\usepackage{cite}
\usepackage{enumerate}

\newtheorem{theorem}{Theorem}
\newtheorem*{theorem*}{Theorem}
\newtheorem{lemma}{Lemma} 
\newtheorem*{lemma*}{Lemma} 
\newtheorem{corollary}{Corollary} 
\newtheorem*{corollary*}{Corollary}

\theoremstyle{definition} 
\newtheorem{remark}{Remark}
\newtheorem*{remark*}{Remark}

\newtheorem*{definition*}{Definition}

\newtheorem*{example*}{Example}

\newtheorem*{problem*}{Problem}

\def\cl#1{\skew3\overline{#1}}
\def\clx#1#2{\skew3\overline{#1}^{\raise.5pt\hbox{\scriptsize$\,#2$\!}}}

\def\clrx#1#2{\skew3\overline{#1\hspace{-0.1em}}^{\raise.5pt\hbox{\,\scriptsize$#2$}}}

\def\clax#1#2#3{\skew6\overline{#1_{\rlap{$\scriptstyle 
#2$}}}^{\raise.5pt\hbox{\hspace{0.2667em}\scriptsize$#3$}}\!%\hphantom{\hbox{$\scriptstyle #2$}}
}
\def\clarx#1#2#3{\skew7\overline{\,#1_{\rlap{$\scriptstyle 
#2$}}}^{\raise.5pt\hbox{\hspace{0.2667em}\scriptsize$#3$}}\!%\hphantom{\hbox{$\scriptstyle #2$}}
}
\def\Cl#1^#2{\cl{#1\kern-.09em}\kern.09em\vphantom{#1}^{#2}}

\let\le\leqslant
\let\ge\geqslant

\begin{document}

{
\renewcommand*{\thefootnote}{$\star$}

\title[Closed Discrete Bases in Countable-Dimensional TVS]{Closed Discrete Bases\\in Countable-Dimensional\\ Topological Vector Spaces}
\footnotetext[0]{This work was financially supported by the Russian Science Foundation, grant~22-11-00075-P.} 

}

\author{Ol'ga Sipacheva}

\begin{abstract}
Let $K$ be either a complete nondiscrete second-countable valued field (e.g., $\mathbb R$ or $\mathbb C$) 
or a finite field. 
It is proved that any countable-dimensional topological vector space over $K$ has a closed 
discrete basis. 
\end{abstract}



\address{Department of General Topology and Geometry, Faculty of Mechanics and  Mathematics, 
M.~V.~Lomonosov Moscow State University, Leninskie Gory 1, Moscow, 199991 Russia}

\email{ovsipa@gmail.com}

\maketitle

\section{Introduction}

Let $L$ be a countable-dimensional Hausdorff topological vector space over a topological field $K$, 
and let $B$ be any basis of $L$. The topology of $L$ may induce an arbitrary Tychonoff topology on $B$. Indeed, 
any topological $K$-vector space $V$ can be treated as a topological universal algebra with continuous signature 
(see \cite{Choban} or \cite{Choban2}). In the terminology of \cite{Choban2}, the signature of $V$ is the 
topological sum of $K$ (each $\lambda\in K$ represents a unary operation) and two singletons $\{o_0\}$ and 
$\{o_2\}$ (each $o_i$ represents an $i$-ary operation). The topological-algebraic structure of $V$ is determined 
by the continuous maps $e_0\colon \{o_0\}\to V$ defined by $e_0(o_0)=\mathbf 0$ (choosing the zero element of $V$), 
$e_1\colon K\times V\to V$ defined by $(\lambda, v)\mapsto \lambda v$ (scalar multiplication), and $e_2\colon 
\{o_2\}\times V\times V\to V$ defined by $(o_2, v, u)\mapsto v+u$ (vector addition). The countable set $B$ 
endowed with any Tychonoff topology $\tau$ is zero-dimensional, i.e., has a base consisting of clopen sets. 
According to \cite[Property~7 on p.~115]{Choban} (see also \cite[Corollary~3.5]{Choban2}), the free topological 
$K$-vector space of $(B,\tau)$ is defined and $(B,\tau)$ is topologically embedded in it as a basis. 

However, if $K$ is either a complete second-countable nondiscrete valued field 
(such as $\mathbb R$ or $\mathbb C$) or a finite field, 
then, whatever the induced topology of $B$ and the Hausdorff vector space topology of $L$, there 
always exists another basis which is closed and discrete in $L$. This is the main result of this paper. 

\section{Preliminaries}

Recall that a \emph{topological group} is a group endowed with a \emph{group topology}, i.e., a topology with 
respect to which both group operations 
(multiplication and inversion) are continuous. A \emph{topological field} is defined in a similar way as a field 
with a \emph{field topology}, i.e., a topology with respect to which all operations of the field (addition, multiplication, and both inversions) 
are continuous. Finally, a \emph{topological vector space} over a topological field $K$ is a vector space $L$ over $K$ endowed 
with a topology with respect to which vector addition and scalar multiplication are continuous functions 
$L\times L \to L$ and $K\times L\to L$, respectively. 

In what follows, by $K$ we always denote a topological field, by $L$, a topological vector space over $K$, and by 
$\tau_L$, the topology of $L$. Given $A\subset L$, we use $\langle A\rangle $ to denote the linear span of $A$ in 
$L$. We usually treat $L$ as a topological group (under vector addition). We denote the 
zero elements of $K$ and $L$ by $0$ and $\mathbf 0$, respectively. The notation $\mathbb N$ is used for the set of 
positive integers and $\mathbb R$, for the space of real numbers. For a prime power $q$, we denote the $q$-element 
field by $\mathbb F_q$; we assume that $\mathbb F_q$ is endowed with the discrete topology and 
$q=p^n$, where $p$ is a prime and $n\in \mathbb N$. 

Let $B$ be a basis of $L$. Then every nonzero element $v$ of $L$ can be uniquely (up to the order of terms) 
represented as a linear combination of elements of $B$ with nonzero coefficients in $K$. We refer to this linear 
combination as the \emph{decomposition of $v$ in $B$} and to the number of terms in it as the \emph{length of~$v$ 
in~$B$}. We assume that the length of $\mathbf 0$ in any basis is~$0$. 

A \emph{valuation}, or \emph{absolute value}, on a field $K$ is a function 
$\lvert\boldsymbol\cdot\rvert\colon K\to \mathbb R$ with the following properties:
\begin{enumerate}
\item
$|\lambda|=0$ if and only if $\lambda = 0$; 
\item
$|\lambda + \mu|\le |\lambda| + |\mu|$ for any $\lambda, \mu\in K$; 
\item
$|\lambda \mu|\le |\lambda|\cdot |\mu|$ for any $\lambda, \mu\in K$. 
\end{enumerate}
A field equipped with a valuation is called a \emph{valued field}. Every valuation on a field $K$ 
determines the metric $d$ on $K$ defined by $d(\lambda, \mu)=|\lambda-\mu|$, and this metric generates 
a field topology on $K$. When considering a valued field, we assume it to be endowed with this topology. 
A valued field is said to be \emph{complete} if it is complete with respect to the metric determined by the 
valuation. 

A \emph{group norm} on a group $G$ with neutral element $e$ is a function $\lVert\boldsymbol\cdot\rVert\colon G\to 
\mathbb R$ with the following properties: 
\begin{enumerate} 
\item $\|g\|=0$ if and only if $g=e$; 
\item $\|g^{-1}\| = \|g\|$; 
\item $\|gh\|\le \|g\|+\|h\|$ (the triangle inequality for norms). 
\end{enumerate} 
We say that the topology of a topological group $G$ is \emph{generated by a group norm} $\lVert \boldsymbol\cdot 
\rVert$ if the sets $U_\varepsilon=\{x\in G: \|x\|< \varepsilon\}$, $\varepsilon > 0$, form a \emph{neighborhood base at $e$} 
for the topology of~$G$, i.e., every neighborhood of $e$ contains $U_\varepsilon$ for some $\varepsilon>0$.

Any group norm $\lVert \boldsymbol\cdot\rVert$ on any group $G$ generates the metric $d$ on $G$ defined by 
$d(g,h)=\|g^{-1}h\|$. This metric is left-invariant, i.e., $d(xg, xh)=d(g,h)$ for any $x,g,h\in G$, and 
satisfies the condition $d(g^{-1}, h^{-1})= d(g, h)$ for all $g,h\in G$. We denote the completion of the metric 
space $(G,d)$ by $\widehat G$ and call it the \emph{completion of $G$ with respect to the norm}~$\lVert 
\boldsymbol\cdot\rVert$. It is well known that the multiplication and inversion group operations extend 
continuously to  $\widehat G$, so that $\widehat G$ is a topological group containing $G$ as a subgroup (see, 
e.g.,~\cite[pp.~352--353]{THJ}); note that the topology of $\widehat G$ is generated by a metric extending~$d$.

Recall that a family $\mathscr B$ of open subsets of a topological space $X$ is called a \emph{base} for the 
topology of $X$ if, given any $x\in X$ and any neighborhood $U$ of $x$, there exists a $V\in \mathscr B$ such that 
$x\in V\subset U$. A space $X$ is said to be \emph{second-countable} if there exists an at most countable base for its topology. 
Dropping the requirement that the elements of $\mathscr B$ be open, we obtain the definition of a network: 
a \emph{network} for $X$ is a family $\mathscr N$ of subsets of $X$ such that, given any $x\in X$ and any neighborhood 
$U$ of $x$, there exists an $N\in \mathscr N$ for which $x\in N\subset U$. It is well known that any Tychonoff 
space with a countable network admits a coarser second-countable Tychonoff topology~\cite{Arhangelskii66}. 

All notions used in this paper and related to topological and metric spaces can be found in~\cite{Engelking}, and 
those related to topological groups and vector spaces, in~\cite{Bourbaki} or \cite{AT}. We assume all topological 
groups and vector spaces under consideration to be Hausdorff.

\section{The Main Theorem}

As mentioned, the main result of this paper is the following theorem.

\begin{theorem}
Let $K$ be either a complete second-countable nondiscrete valued field or a finite field. 
Then any countable-dimensional Hausdorff topological vector space over $K$ has a closed discrete basis. 
\end{theorem}

To prove the main theorem, we need two auxiliary lemmas. We begin with the case of a nondiscrete field~$K$. 

\begin{lemma}
\label{l1}
Let $L$ be a countable-dimensional topological vector space over a complete second-countable nondiscrete valued 
field $K$. Then there exists a continuous group norm $\lVert\boldsymbol \cdot\rVert$ on $L$ (treated as a 
topological group) and a basis $B'$ in $L$ such that the elements of $B'$ form a sequence converging to $\mathbf 0$ 
in the topology generated by $\lVert\boldsymbol \cdot\rVert$ on~$L$. Moreover, $L$ is incomplete with respect 
to~$\lVert\boldsymbol \cdot\rVert$. 
\end{lemma}

\begin{proof}
Let $B=\{b_1, b_2, \dots\}$ be any basis of $L$, and let $\mathscr B$ be an at most countable base for the 
topology of $K$. Take any element $v$ of $L$. It can be represented as $\sum_{i\le k} \lambda_i b_i$, where $k\in 
\mathbb N$ and $\lambda_i\in K$ for $i\le k$, and it follows from the continuity of scalar multiplication in $L$ 
that any neighborhood of $v$ contains a set of the form 
$\bigl\{\sum_{i\le k} \mu_i b_i: \mu_i\in B_i$ for $i\le k\bigr\}$,
where $B_i\in \mathscr B$ and $\lambda_i\in B_i$ for $i\le k$. It follows that all such sets form a network of $L$, which 
is obviously countable. Note that the vector space $L$ is the union of its finite-dimensional subspaces 
$$
L_n=\{\lambda_1 b_1 + \dots +\lambda_n b_n: \lambda_i \in K \text{ for }i\le n\}.
$$ 
Since $K$ is complete, 
 each $L_n$ is closed \cite[Sec.~I.2.3, Corollary~1]{Bourbaki} and hence nowhere dense \cite[Sec.~I.2.1]{Bourbaki} 
in $\tau_L$; in particular, the interior of each of them is empty. According to Lemma~8.1.12 of \cite{AT}, 
the additive group $L$ admits a coarser second-countable Hausdorff group topology $\tau$ in which all $L_n$ are 
closed. Since the interior of every $L_n$ remains empty in $\tau$, it follows that the $L_n$ are nowhere dense in 
$\tau$ and hence the topological group $(L,\tau)$ is incomplete by the Baire category theorem. The group topology 
$\tau$, being second-countable, is generated by a group norm $\lVert\boldsymbol\cdot\rVert$ (see, e.g., 
\cite{Graev50} or \cite[Corollary~3.3.13]{AT}). Let $(\xi_n)_{n\in\mathbb N}$ be a sequence  of pairwise distinct 
elements of $K$ converging to $0$ in $K$ (it exists because $K$ is metrizable and nondiscrete). For each $i\in \mathbb N$, 
the sequence $(\xi_n b_i)_{n\in \mathbb N}$ converges to $\mathbf 0$ in $(L,\tau_L)$ and hence in $(L,\tau)$. 
Choose numbers $j(i)\in \mathbb N$ for $i\in \mathbb N$ so 
that $\|\xi_k b_i\|<1/i$ for $k\ge j(i)$ and let $b'_i= \xi_{j(i)} b_i$ for $i\in \mathbb N$. Clearly, the 
sequence $(b'_i)_{i\in \mathbb N}$ converges to $\mathbf 0$ in $(L,\tau)$. Therefore, $\mathbf 0$ 
is the only limit point of the set $B'=\{b'_i: i\in \mathbb N\}$ in $G$ and even in the completion $\widehat L$ 
of $L$ with respect to $\lVert\boldsymbol\cdot\rVert$. Obviously, $B'$ is a basis in~$L$. 
\end{proof}

Now suppose that $K$ is finite (and discrete), i.e., $K=\mathbb F_q$ for a prime 
power~$q$. 

\begin{lemma}
\label{l2}
Given any countable-dimensional topological vector space $L$ over a finite field $\mathbb F_q$, 
there exists a continuous group norm $\lVert\boldsymbol\cdot\rVert$ on $L$ and a basis $B'$ 
of $L$ such that $B'$ is discrete in the topology generated by $\lVert\boldsymbol\cdot\rVert$ and either $B'$ 
is closed in the completion $\widehat L$ of $L$ with respect to $\lVert\boldsymbol\cdot\rVert$ or $\mathbf 0$ is its 
only limit point in~$\widehat L$. Moreover, the group $L$ is incomplete with respect 
to~$\lVert\boldsymbol\cdot\rVert$. 
\end{lemma}

\begin{proof}
Let $B$ be any basis of $L$; we number its elements as $b_1,b_2, \dots$\,. Any countable 
topological space, in particular, $L$, has a countable network. According to~\cite[Lemma~8.1.12]{AT} 
and~\cite{Graev50}, there exists a continuous group norm $\lVert\boldsymbol\cdot\rVert$ on~$L$, which 
generates a group topology $\tau$ on~$L$. 

Consider the set $B'=\{b'_1, b'_2, \dots\}$ defined recursively as follows: 
\begin{enumerate} 
\item[(i)] 
$b'_1= b_1$; 

\item[(ii)] 
if $n\in \mathbb N$ and $b'_1, b'_2, \dots, b'_{n}$ are already defined, then $b'_{n+1}$ is a linear 
combination $\lambda_1 b'_1 +\lambda_2 b'_2 + \dots+ \lambda_n b'_{n}+\lambda_{n+1} b_{n+1}$, where 
$\lambda_{n+1}\ne 0$, of minimum norm (if there are several linear combinations of minimum norm, then we take any 
of them). Such a $b'_{n+1}$ exists because the number of linear combinations of elements of a finite set over a 
finite field is finite.
\end{enumerate}

It is easy to see that $B'$ has the following properties.
\begin{enumerate} 
\item[(iii)]
For each $n$,  $\langle \{b'_1, \dots, b'_n\}\rangle = \langle \{b_1, \dots, b_n\}\rangle $.
\item[(iv)]
The set $B'$ is a basis in $L$. 
\item[(v)]
If $n, i_1,\dots, i_n\in \mathbb N$ and $i_1<\dots <i_n$, then $\|b'_{i_n}\|\le 
\|\lambda_1 b'_{i_1}+\dots +\lambda_n b'_{i_n}\|$ for any $\lambda_1, \dots, \lambda_n\in 
\mathbb F_q\setminus\{0\}$. 
\end{enumerate} 

Property (i) is proved by simple induction. For $n=1$, we have $\{b'_n\}=\{b_n\}$ and hence 
$\langle \{b'_n\}\rangle =\langle \{b_n\}\rangle$. Suppose that $n>1$ and 
$\langle \{b'_1, \dots, b'_k\}\rangle = \langle \{b_1, \dots, b_k\}\rangle$ for all $k<n$. 
By definition 
$$
b'_n=\lambda_1 b'_1+ \lambda_2b'_2+ \dots +\lambda_{n-1}b'_{n-1}+\lambda_n b_n,
$$ 
where $\lambda_n\ne 0$. By the induction hypothesis 
$b'_n\in \langle \{b_1, \dots, b_n\}\rangle$ and hence $\langle \{b'_1, \dots, b'_n\}\rangle \subset 
\langle \{b_1, \dots, b_n\}\rangle$. On the other hand, 
$$
-\lambda_n b_n= \lambda_1 b'_1+ \lambda_2b'_2+ \dots +\lambda_{n-1}b'_{n-1}-b'_n
$$ 
and therefore  
$b_n\in \langle \{b'_1, \dots, b'_{n}\}\rangle$, whence 
$\langle \{b_1, \dots, b_n\}\rangle \subset \langle \{b'_1, \dots, b'_n\}\rangle$. This proves~(i).

To prove (ii), it suffices to note that the decomposition of every $b'_n$ in $B$  
contains $b_n$ with nonzero coefficient and does not contain $b_k$ with $k>n$, because 
$$
b'_n=\lambda_1 b'_1+ \lambda_2b'_2+ \dots +\lambda_{n-1}b'_{n-1}+\lambda_n b_n
$$ 
and 
$b'_1, b'_2, \dots, b'_{n-1}\in \langle \{b_1, \dots, b_{n-1}\}\rangle$ by (i). Thus, no vector $b'_n\in B'$ 
can be represented as a linear combination of vectors $b'_k$ with $k<n$, which implies the  
linear independence of~$B'$.

Let us prove~(iii). Suppose that $n, i_1,\dots, i_n\in \mathbb N$, $i_1<\dots <i_n$, and 
$\lambda_1,\dots,\lambda_n\in \mathbb F_q$. By the definition of $b'_n$ we have 
$$
b'_n=\mu_1 b'_1+ \mu_2b'_2+ \dots +\mu_{n-1}b'_{n-1}+\mu_n b_n 
$$
for some $\mu_1, \dots, \mu_{n-1}\in \mathbb F_q$ and $\mu_n\in \mathbb F_q\setminus \{0\}$; therefore, 
$$
\lambda_1 b'_1+ \lambda_2b'_2+ \dots +\lambda_n b'_n = 
\nu_1 b'_1+\nu_2 b'_2 + \dots + \nu_{n-1} b'_{n-1}+ \nu_n b_n,
$$
where $\nu_i = \lambda_i + \lambda_n\cdot \mu_i$ for $i\le n-1$ and $\nu_n=\lambda_n\cdot \mu_n\ne 0$. 
 By definition the norm of $b'_n$ is no greater than the norm of any linear combination of 
$b'_1, \dots, b'_{n-1}$ and $b_n$ in which the coefficient of $b_n$ is nonzero; this 
implies the desired inequality.  

\smallskip

It follows from property (iii) that any sequence $(b'_{n_k})_{k\in \mathbb N}$ of pairwise distinct elements 
of $B'$ which is Cauchy with 
respect to the group norm $\lVert\boldsymbol\cdot\rVert$ converges to~$\mathbf 0$ in the norm topology $\tau$. 
Indeed, by the definition of a 
Cauchy sequence, for any $\varepsilon >0$, there exists an $N\in \mathbb N$ such that 
$\|b'_{n_{k}}-b'_{n_N}\|< \varepsilon$ for any $k\ge N$. Since all integers $n_k$, $k\in \mathbb N$, 
are pairwise distinct, it follows that only finitely many of them are smaller than $n_N$ and hence there exists 
an $M\ge N$ such that $n_k>n_N$ for any $k\ge M$. 
But if $n_k > n_N$, then 
$\|b'_{n_k}\|\le \|b'_{n_k}-b'_{n_N}\|$ by property~(iii) of $B'$. Therefore, 
$\|b'_{n_k}\|<\varepsilon$ for all~$k > M$. 

Thus, in the completion $\widehat L$ of $L$ with respect to $\lVert\boldsymbol\cdot\rVert$, only $\mathbf 0$ 
can be a limit point of $B'$. 

Being countable and nondiscrete, the group $L$ is incomplete with respect to $\lVert\boldsymbol\cdot\rVert$ by 
the Baire category theorem. 
\end{proof}

\begin{proof}[of main theorem]
Let $K$ be either a complete second-countable non\-dis\-crete valued field or a finite field, and let 
$L$ be a countable-dimensional topological vector space over $K$. According to Lemmas~\ref{l1} and~\ref{l2}, 
there exists a continuous group norm $\lVert\boldsymbol\cdot\rVert$ on the additive group $L$ 
and a basis $B'=\{b'_n: n\in \mathbb N\}$ of the vector space $L$ such that only $\mathbf 0$ can be a limit 
point of $B'$ in the completion $\widehat L$ of $L$ with respect to~$\lVert\boldsymbol\cdot\rVert$. Note that 
the set $B' \cup \{\mathbf 0\}$ is closed in~$\widehat L$. 

Consider the function $\max\colon L\to \{0\}\cup \mathbb N$ defined by setting $\max(\mathbf 0) = 0$ and 
$$
\max (\lambda_1 b'_{n_1} + \dots + \lambda_k b'_{n_k}) = n_k
\quad \text{for $n_1< \dots < n_k$ and $\lambda_1, \dots, \lambda_k\in \mathbb F_q\setminus \{0\}$}.
$$
In particular, $\max(b'_n)=n$ for each $b'_n$.

According to Lemmas~\ref{l1} and~\ref{l2}, $L$ is incomplete with respect to $\lVert\boldsymbol\cdot\rVert$, i.e., 
$L\ne \widehat L$. Take $a\in \widehat L\setminus L$ 
and choose a sequence $(a_n)_{n\in \mathbb N}$ of pairwise distinct nonzero elements of $L$ converging to $a$ in 
the metric of~$\widehat L$. 

Let $f\colon \mathbb N\to \mathbb N$ be the function defined by 
$$ 
f(i) = \max(a_i) \qquad\text{for $i\in \mathbb N$}. 
$$
Note that $f$ is unbounded. Indeed, otherwise, there exists an $N\in \mathbb N$ such that $\max(a_i)\le N$ for all 
$i\in \mathbb N$, which means that every $a_i$ belongs to the linear span $L_N$ of the set $\{b'_1,\dots, b'_N\}$. 
If $K$ is finite, then so is $L_N$ and the sequence $(a_n)_{n\in \mathbb N}$ cannot converge to any point not 
belonging to $L_N\subset L$. Suppose that $K$ is infinite. By assumption it is nondiscrete and complete. 
According to \cite[Sec.~I.2.3, Theorem~2]{Bourbaki}, $L_N$ is topologically isomorphic to the complete vector 
space $K^N$, so that the sequence $(a_n)_{n\in \mathbb N}$ cannot converge to a point in $\widehat L\setminus L$ 
in this case either. 

Passing to a subsequence if necessary, we can assume that the function $f(i)$ strictly increases and $f(1)\ge 2$. 
For the $n$th iterate of $f$, where $n\in \mathbb N$, we use the standard notation $f^n$. 
We also set $f^{-1}(1)=0$ and $f^0(1)=1$ for convenience. We put 
\begin{gather*} \allowdisplaybreaks
b''_0\! = b'_1,\\ 
b''_1\! = b'_1\!+a_{f^0(1)},\ b''_2\! = b'_2\!+a_{f^0(1)},\ \dots, \ b''_{f(1)-1}\!= b'_{f(1)-1}\!+a_{f^0(1)},\\ 
b''_{f(1)}\! = b'_{f(1)}\! + a_{f(1)},\ b''_{f(1)+1}\! = b'_{f(1)+1}\! + a_{f(1)},\ \dots, \ b''_{f^2(1)-1}\! = 
b'_{f^2(1)-1}\! + a_{f(1)},\\ 
b''_{f^2(1)}\! = b'_{f^2(1)}\! + a_{f^2(1)}, \ b''_{f^2(1)+1}\! = b'_{f^2(1)+1}\! + 
a_{f^2(1)},\ \dots, \ b''_{f^3(1)-1}\! = b'_{f^3(1)-1}\! + a_{f^2(1)},\\ 
\dots,\\ 
b''_{f^k(1)}\! = b'_{f^k(1)}\! + 
a_{f^{k}(1)}, \ \dots, \ b''_{f^{k+1}(1)-1}\! = b'_{f^{k+1}(1)-1}\! + a_{f^{k}(1)},\\ \dots\,. 
\end{gather*} 

Let us show that the set $B'' = \{b''_0, b''_1, \dots\}$ is a basis of $L$. 
First, we check its linear independence. Take any linear combination $\lambda_{i_1}b''_{i_1} + \dots + 
\lambda_{i_k}b''_{i_k}$, where $i_1 < \dots < i_k$ and $\lambda_{i_j}\in K\setminus\{0\}$ for $1\le j\le k$. 
We must show that it does not vanish. 

 For $n\ge -1$, we set 
$$ 
F_n=\bigl\{l\le k: i_l\in \{f^{n}(1), \dots, f^{n+1}(1)-1\}\bigr\}. 
$$ 
Note that every $l\le k$ belongs to $F_n$ for some $n\ge -1$. Let $k$ belong to $F_m$. 
Then the number $r$ of any set $F_r$ containing some $l\le k$ satisfies the inequality  
$r\le m$. Suppose that $l\in F_r$. If $r\ge 0$, then $i_l\ge f^0(1)=1$, so that 
$$
\max(b''_{i_l})=\max(b'_{i_l}+a_{f^r(1)})\quad\text{and}\quad \max(a_{f^r(1)})=f^{r+1}(1) > i_l=\max(b'_{i_l}), 
$$
and if $r=-1$, then $F_r=\{0\}$, $i_l=0$, and $\max(b''_{i_l})=\max(b'_1)=1=f^0(1)=f^{r+1}(1)$. 
Therefore, 
$$
\max(\lambda_{i_l}b''_{i_l}) = f^{r+1}(1)\qquad\text{for any $r\le m$ and $l\in F_r$}.
$$ 
It follows that 
$$
\max(\lambda_{i_l}b''_{i_l}) \le f^m(1)\qquad\text{for all $l\notin F_m$, $l\le k$}.
\eqno{(*)}
$$ 

If $\sum_{i\in F_m} \lambda_{i_l}=0$ in $K$, then $F_m$ contains at least two elements. Therefore, $m\ge 0$ 
and $i_l\ge 1$ for all $l\in F_m$. Thus, 
$$ 
\sum_{l\in F_m}\lambda_{i_l} b''_{i_l} = \sum_{l\in F_m} \lambda_{i_l} (b'_{i_l}+a_{f^{m}(1)}) = 
\sum_{l\in F_m} \lambda_{i_l} b'_{i_l}+ \Bigl(\sum_{l\in F_m} \lambda_{i_l}\Bigr) a_{f^{m}(1)} =
\sum_{l\in F_m} 
\lambda_{i_l}b'_{i_l}. 
$$ 
Hence, in this case, the decomposition of $\sum_{l\in F_m}\lambda_{i_l} b''_{i_l}$ in the basis $B'$ contains 
$b'_{i_k}$ with the nonzero coefficient $\lambda_{i_k}$. Note that $i_k > f^{m}(1)$. Indeed, 
we have $k-1\in F_m$, because $F_m$ contains at least two elements; 
hence  $i_{k-1}\ge f^{m}(1)$ and $i_k > f^{m}(1)$. Since 
$b'_{i_k}$ does not appear in the  decomposition of $\sum_{l\notin F_m,\, l\le k}\lambda_{i_l} b''_{i_l}$ in view of $(*)$, 
it follows that the linear combination 
$$
\lambda_{i_1}b''_{i_1} + \dots + \lambda_{i_k}b''_{i_k}= \sum_{l\notin F_m,\,l\le k}\lambda_{i_l} b''_{i_l} + 
\sum_{l\in F_m}\lambda_{i_l} b''_{i_l}
$$ 
does not vanish.

Suppose that $\lambda=\sum_{l\in F_m}\lambda_{i_l}\ne 0$ and $m>-1$. Then $i_l\ge 1$ for all $l\in F_m$ and 
$$ 
\sum_{l\in F_m}\lambda_{i_l} b''_{i_l} = \sum_{l\in F_m}\lambda_{i_l} (b'_{i_l}+a_{f^{m}(1)}) = \sum_{l\in 
F_m}\lambda_{i_l} b'_{i_l} + \lambda a_{f^{m}(1)}.  
$$ 
We have 
$$ 
\max \Bigl(\sum_{l\in F_m}\lambda_{i_l} b'_{i_l}\Bigr)= i_k \le f^{m+1}(1)-1
$$
and 
$$
\max(\lambda a_{f^m(1)})= \max(a_{f^m(1)})=f^{m+1}(1);   
$$ 
therefore, 
$$ 
\max \Bigl(\sum_{l\in F_m}\lambda_{i_l} b''_{i_l}\Bigr)= \max(\lambda a_{f^{m}(1)}) = f^{m+1}(1). 
$$ 
It follows that the decomposition of $\sum_{l\in F_m}\lambda_{i_l} b''_{i_l}$ in $B'$ contains the element 
$b'_{f^{m+1}(1)}$ with a nonzero  coefficient, while in the decomposition of 
$\sum_{l\notin F_m,\,l\le k} b''_{i_l}$ this element does not appear in view of $(*)$.  
Thus, the linear combination $\lambda_{i_1}b''_{i_1} + \dots + \lambda_{i_k}b''_{i_k}$ does not vanish in the 
case where $\sum_{l\in F_m}\lambda_{i_l}\ne 0$ and $m>-1$ either. 

Finally, if $m=-1$, then $F_m=\bigl\{l\le k:i_l\in \{0\}\bigr\}$ and $b''_{i_k}=b''_0$, so that the linear 
combination under consideration is $\lambda_{i_k}b''_0=\lambda_{i_k} b'_1\ne\mathbf 0$. 

We have proved the linear independence of the set $B''$. Let us show that all elements of $L$ 
are linear combinations of elements of $B''$. It suffices to check that every $b'_n\in B'$ 
is a linear combination of elements of $B''$. We argue by induction on $n$. For $n=1$, this is true. Suppose that 
$n>1$ and we have already represented all $b'_k$, $k<n$, as linear combinations of elements of $B''$. Note 
that $n> f^0(1)$. Let $i$ 
be the greatest integer for which $n\ge f^{i+1}(1)$. Then $i\ge -1$ and $n=f^{i+1}(1)+j$ for some nonnegative 
integer $j< f^{i+2}(1)-f^{i+1}(1)$. Suppose that $j=0$; then $n=f^{i+1}(1)$ and $i\ge 0$, because $n>1$. By construction  
$$ 
b'_{f^i(1)} = b''_{f^i(1)} - a_{f^i(1)}. 
$$ 
Since 
$$
\max(a_{f^i(1)})=f({f^i(1)})=f^{i+1}(1)=n,
$$ 
it follows that 
$$ 
a_{f^i(1)} = \sum_{k=1}^n \lambda_k b'_k,\quad \text{where $\lambda_k\in K$ and $\lambda_n\ne 0$};
$$ 
therefore, 
$$
\lambda_n b'_n= a_{f^i(1)} - \sum_{k=1}^{n-1} \lambda_k b'_k.
$$
For each $k< n$, $b'_k$ is a linear combination of elements of $B''$ by the induction 
hypothesis. Hence $b'_n$ can be expressed as a linear combination of $a_{f^i(1)}$ and elements of $B''$. 
On the other hand, 
$$ 
a_{f^i(1)}=b'_{f^i(1)} - b''_{f^i(1)}.
$$
Since $f^i(1)<n=f^{i+1}(1)$, it follows by the induction hypothesis that $b'_{f^i(1)}$ can be expressed as a linear 
combination of elements of $B''$; therefore, so can  $a_{f^i(1)}$ and hence~$b'_n$. 

Now suppose 
that $j> 0$. Then  $n>f^{i+1}(1)$ and  by the induction 
hypothesis  $b'_{f^{i+1}(1)}$ can be represented as a linear combination of elements of $B''$. Since 
$$
b''_{f^{i+1}(1)} = b'_{f^{i+1}(1)}+a_{f^{i+1}(1)},
$$ 
it follows that so can $a_{f^{i+1}(1)}$. By construction 
$$
b''_{f^{i+1}(1)+j} = b'_{f^{i+1}(1)+j}+a_{f^{i+1}(1)}; 
$$
therefore, $b'_{f^{i+1}(1)+j} = b'_n$ can be represented as a linear combination of elements of~$B''$.

Thus, $B''$ is a basis. Let us show that it is closed and discrete in $L$. Arguing by contradiction, suppose 
that there exists a point $b$ in $L$ which is limit for $B''$ in the group norm topology. Then there exists a 
sequence $(b''_{n_k})_{k\in \mathbb N}$ of pairwise distinct elements of $B''$ converging to $b$ in $L$ 
with the group norm topology and hence in $\widehat L$; we assume that $n_k\ge 1$ for all $k\in \mathbb N$. 
 By construction this sequence has the form $(b'_{n_k}+a_{m_k})_{k\in \mathbb N}$, 
where all $n_k$ are pairwise distinct and every $a_i$ may occur only finitely many times in the sequence. 
Passing to a subsequence if needed, 
we can assume that all $m_k$ are pairwise distinct as well. Note that $(b'_{n_k})_{k\in \mathbb N}$ is a Cauchy 
sequence with respect to the metric of $\widehat L$, because so are $(b'_{n_k}+a_{m_k})_{k\in \mathbb N}$ and $(a_{m_k})_{k\in \mathbb N}$. Therefore, 
$b'_{n_k}\to \mathbf 0$ as $k\to \infty$ by Lemma~\ref{l2}. Since $a_{m_k}\to a$ as $k\to \infty$ in $\widehat L$ and 
$a\notin L$, 
it follows that $b''_{n_k}\to a\ne b$. This contradiction shows that $B''$ has no limit points in $L$ with 
the group norm topology $\tau$. Hence $B''$ has no limit points in $L$ with the finer topology $\tau_L$ 
either, i.e., $B''$ is closed and discrete in~$L$. 
\end{proof}

\begin{corollary}
Any countable-dimensional real or complex  Hausdorff topological vector space has a closed discrete basis. 
\end{corollary}

\begin{remark}
Uncountable-dimensional topological vector spaces do not necessarily contain closed discrete bases. For 
example,  given any prime power $q$ and any cardinal $\kappa$, the Tychonoff product $\mathbb F_q^\kappa$ is a 
compact topological $\mathbb F_q$-vector space and hence cannot contain infinite closed discrete 
sets. This product is finite for finite $\kappa$, while for infinite $\kappa$, its cardinality is at least the 
continuum. Also, if $K$ is any second-countable metrizable topological field, $\kappa$ is a cardinal, and $L$ is the 
direct sum of $\kappa$ homeomorphic copies of $K$ with the Tychonoff product topology (that is, $L$ is a subspace 
of the product space $K^\kappa$ consisting of all points with only finitely many nonzero coordinates; such a space 
is called a \emph{$\sigma$-product} of copies of $K$), then $L$ is Lindel\"of~\cite[Proposition~3]{Corson} and 
hence cannot contain uncountable closed discrete sets.
\end{remark}

\end{document}